\documentclass[12 pt]{amsart}  
\usepackage{amsthm, amsmath, amscd, amssymb, latexsym, stmaryrd, color}

\usepackage[all]{xypic}
\usepackage{color}
\usepackage{fullpage}
\definecolor{hot}{RGB}{65,105,225}
\usepackage[pagebackref=true,colorlinks=true, linkcolor=hot ,  citecolor=hot, urlcolor=hot]{hyperref}

\theoremstyle{plain}

\newtheorem{theorem}{Theorem}[section]
\newtheorem{lemma}[theorem]{Lemma}

\newtheorem{prop-def}[theorem]{Proposition-Definition}
\newtheorem{corollary}[theorem]{Corollary}

\theoremstyle{definition}

\newtheorem{definition}[theorem]{Definition}

\newtheorem{example}[theorem]{Example}

\newtheorem{question}[theorem]{Question}
\theoremstyle{remark}
\newtheorem{remark}[theorem]{Remark}

\usepackage[mathscr]{eucal}
\usepackage{graphics, graphpap}
\usepackage{array, tabularx, longtable}
\usepackage{color}
\numberwithin{equation}{section}

\def\Mor{\mathrm{Mor}}

\def\Spec{\mathrm{Spec}\,}
\def\Spf{\mathrm{Spf}}

\title{Curve selection lemma in arc spaces}  
\author{Nguyen Hong Duc}
\address{$^{\dag}$TIMAS, Thang Long University, \newline \indent Nghiem Xuan Yem, Hanoi, Vietnam.} 
\email{duc.nh@thanglong.edu.vn}
\thanks{}

\begin{document}           

\begin{abstract}
We first generalize a curve selection lemma for Noetherian schemes and apply it to prove a version of Curve Selection Lemma in arc spaces, answering affirmatively a question by Reguera. Furthermore, thanks to a structure theorem of Grinberg, Kazhdan and Drinfeld, we obtain other versions of Curve Selection Lemma in arc spaces.
\end{abstract}

\maketitle                 

\section{Introduction}
Curve Selection Lemma is shown to be a very useful tool in many geometric situations in algebraic, analytic and semi-algebraic geometry. The classical version of Curve Selection Lemma was achieved by Milnor \cite{Mi68}. Let $X$ be a semi-algebraic subset in $\mathbb R^n$ and $x$ be a point in the closure $\bar{X}$ of $X$. Then there exists a Nash curve (analytic and semi-algebraic curve) 
$$\phi:[0,\varepsilon)\to \mathbb R^n$$
such that $\phi(0)=x$ and $\phi(t)\in X$ for all $t\in (0,\varepsilon)$.  
In algebraic geometry a version of Curve Selection Lemma for varieties, which can be proved by using a cutting method, is stated as follows. Let $X$ be a scheme of finite type over a field $k$. If a non-isolated point $x$ is in the Zariski closure $\bar A$ of a constructible subset $A$, then there is a non-constant morphism
$$\alpha\colon \Spec(k_x[[t]]) \to \bar A$$
sending the closed point to $x$ and the generic point to a point in $A$. If $k_x=k$ is equal to $\Bbb C$ or $\Bbb R$ the parametrization can be chosen
convergent, or algebraic.

We are interested in the study of Curve Selection Lemma in the arc spaces. The difficulty is that the arc spaces are of infinite dimension and it is widely known that a plain formulation of Curve Selection Lemma in infinite dimensional algebraic geometry as stated above is not true in genreral as the following example shows. 

Consider $A:=V\left(\left\{x_1-x_n^n\right\}_{n \in \mathbb{N}}\right)$. Let $a$ be equal to the origin. There is no morphism
$$
\alpha: \operatorname{Spec}(\mathbb{K}[[t]]) \rightarrow A
$$
such that $\alpha(0)$ is equal to the origin $a$ and such that the image of the generic point is not the origin, since otherwise the order of the formal power series $x_1(\alpha(t))$ must be finite and would be divisible by $n$ for all positive integers $n$.

The first version of Curve Selection Lemma for arc spaces, due to Reguera in \cite[Corollary 4.8]{Reg06} is of the following form. Let $X$ be an algebraic variety and let $N$ and $N^{\prime}$ two irreducible subsets of the arc space $X_\infty$ such that $\bar{N} \subsetneq N^{\prime}$. Suppose that $N$ is generically stable  (e.g.  weakly stabe in the sense of Denef-Loeser \cite{DL99}, see Definition \ref{def31}) with the residue field $K$. Then there is a finite algebraic extension $K \subset L$ and a morphism
$$
\alpha\colon \operatorname{Spec}(L[[t]]) \rightarrow X_\infty
$$
whose special point is sent to the generic point of $N$ and such that the image of the generic point $\operatorname{Spec}\left(L((t))\right)$ falls in $N^{\prime} \backslash \bar{N}$.

This version of Curve Selection Lemma has many applications in the study of arc spaces of
algebraic varieties (see for example \cite{JP12a, JP12, Lej80, Lej99, LR12, P, Reg06}). Especially it plays an essential role in the proofs of  of Nash problem for surfaces in \cite{JP12} and for terminal singularities in \cite{dFD16}.

In this paper we introduce stronger versions of Curve Selection Lemma. More concretely, we prove two versions of the Curve Selection Lemma  in arc spaces under the assumption that either the closure of the set $\{x\}$ is generically stable or $x$ is a {\em non-degenerate $k$-arc} (i.e. the corresponding morphism can not factor through the singular locus of the considered variety, see Section \ref{sec31}) and $A$ is generically stable. The first version (Theorem \ref{thm32}) answers affirmatively a question by Reguera in \cite{Reg06}. For the proof of the second version  (Theorem \ref{main}), we need to generalize the structure theorem of Grinberg-Kazhdan and Drinfeld to generically stable subsets. Precisely, we prove that the formal neighbourhood of a generically stable subset of an arc space at a non-degenerate $k$-arc is isomorphic to the product of a local adic Noetherian formal $k$-scheme and an infinitely dimensional affine formal disk. 

\section{Curve selection lemma in Noetherian schemes}
Throughout this note, $k$ is a field. If $x$ is a point of a $k$-scheme then $k_x$ denotes the residue field of $x$. In this section we prove a strong versions of Curve Selection Lemma for Noetherian schemes which generalizes the version stated in the introduction.
\begin{theorem}\label{thm12}
Let $X$ be an irreducible Noetherian $k$-scheme and let $z$ be its generic point. Then for any point $x$ of $X$ there exist an extension $K$ of $k_x$ and an arc $\gamma\colon\mathrm{Spec}(K[[t]])\to X$ which maps the closed point to $x$ and the generic point to $z$.
\end{theorem}
\begin{proof}
Consider the blowing-up $h:Y\to X$ of $X$ along the closure $Z$ of $\{x\}$ in $X$. Let $E\subset Y$ be a prime exceptional divisor which dominates $Z$. Let $y$ be the generic point of $E$ and let $\mathcal{O}_{Y,y}$ the localization of $\mathcal{O}_{Y}$ at $y$. Since $\mathcal{O}_{Y}$ is Noetherian and since $E$ is a divisor on $Y$, $\mathcal{O}_{Y,y}$ is a Noetherian ring of dimension $1$. It follows that the normalization of the completion of $\mathcal{O}_{Y,y}$ is isomorphic to $K[[t]]$, where $K=k_y$.  Let $\phi$ be the arc defined by the following composition of injective morphisms
$$ {\mathcal O}_{X}\to {\mathcal O}_{X,x}\to {\mathcal O}_{Y,y}\to \hat{\mathcal O}_{Y,y}\to K[[t]],$$
where the last morphism is the normalization. Then $\phi(0)=x$ since $ {\mathcal O}_{X,x}\to K[[t]]$ is a morphisms of local rings, and $\phi(\eta)=z$ due to the injectivity of the morphism $ {\mathcal O}_{X}\to K[[t]]$.
\end{proof}
The following corollary is a direct consequence of the theorem where we consider the closure of $\{y\}$ instead of $X$.
\begin{corollary}
Let $X$ be a Noetherian $k$-scheme. Let $x,y$ be two points of $X$ such that $x$ is a specilization of $y$. Then  there exist an extension $K$ of $k_x$ and an arc $\gamma\colon\mathrm{Spec}(K[[t]])\to X$ which maps the closed point to $x$ and the generic point to $y$.
\end{corollary}
\begin{corollary}\label{coro22}
Let $X$ be an irreducible Noetherian $k$-scheme of positive dimension. Let $x$  be a non-isolated $k$-point of $X$ and $Z$ a strictly closed subset of $X$. Then  there exists an arc $$\gamma\colon\mathrm{Spec}(k[[t]])\to X$$ which maps the closed point to $x$ and the generic point outside $Z$.
\end{corollary}
\begin{proof}
Since $Z$ is a closed subset of $X$, we can find another closed subset $Y$ of $X$ of dimension $1$ containing $x$ such that $Y\cap Z\subset \{x\}$. Then, as in the proof of Theorem \ref{thm12}, one has a morphism
 $$ {\mathcal O}_{Y}\to {\mathcal O}_{Y,y}\to \hat{\mathcal O}_{Y,y}\to k[[t]],$$
which defines the expected arc.
\end{proof}
\section{Curve selection lemma in arc spaces}
\subsection{Generically and weakly stable subsets of the space of arcs}\label{sec31}
Let $X$ be a $k$-variety. For any $n$ in $\mathbb N$, denote by $X_n$ (or, $J_nX$) the $k$-scheme of $n$-jets of $X$, which represents the functor from the category of $k$-algebras to the category of sets sending a $k$-algebra $A$ to $\Mor_{k\text{-schemes}}(\Spec(A[t]/A(t^{n+1})),X)$. For $m\geq n$, the truncation $k[t]/(t^{m+1})\to k[t]/(t^{n+1})$ induces a morphism of $k$-schemes 
$$\pi_n^m:X_m\to X_n.$$ 
We call the projective limit 
$$X_\infty:=\varprojlim X_n$$ 
the {\it arc space of $X$}. For any field extension $K\supseteq k$, the $K$-points, or {\em $K$-arcs} of $X_\infty$ correspond one-to-one to the $K[[t]]$-points of $X$. A $K$-arc $x$ of $X_\infty$ is said to be {\em non-degenerate} if its corresponding morphism $\operatorname{Spec}(K[[t]]) \rightarrow X_\infty$ can not factor through the singular locus $\mathrm{Sing}X$ of $X$.

For each $n\in \mathbb N$ we denote by $\pi_n$ (or, $\pi_{n,X}$) the natural morphism $X_\infty\to X_n.$ Let $A$ be a subset of the arc space $X_\infty$. The set $A$ is said to be {\it weakly stable at level $n$}, for some $n$ in $\mathbb N$, if $A$ is a union of fibers of $\pi_n\colon X_\infty\to X_n$; the set $A$ is said to be {\it weakly stable} if it is weakly stable at some level.
\begin{definition}\label{def31}
A locally closed subset $N$ of $X_{\infty}\setminus {(\mathrm{Sing}X)_\infty}$ will be called {\it generically stable} if there exists an open affine subscheme $W$ of $X_{\infty}$, such that $N \cap W$ is weakly stable.
\end{definition}
\begin{remark} 
Our notion of generic stability is slightly different from that of \cite[Definition 3.1]{Reg06}. They coincide if the base field $k$ is perfect by \cite[Theorem 4.1]{Reg06}.
\end{remark}

\begin{lemma} \cite[Corollary 4.6]{Reg06}\label{lm21}
 Let $N$ be an irreducible generically stable subset of $X_{\infty}$, and let $z$ be its generic point. Then:
\begin{itemize}
\item[(i)] the ring $\widehat{\mathcal{O}_{X_{\infty}, z}}$ is Noetherian;
\item[(ii)]  if $N^{\prime}$ is an irreducible subset of $X_{\infty}$ such that $N^{\prime} \supset N, N \neq N^{\prime}$, then $\widehat{\mathcal{O}_{N^{\prime}, z}}$ is a Noetherian local ring of dimension $\geqslant 1$.
\end{itemize}
\end{lemma}
\begin{proof}
The proof of \cite[Corollary 4.6]{Reg06} works with our definition of generically stable subsets of arc spaces.
\end{proof}
Most of the locally closed subsets considered in the literature are generically stable. Examples are cylindrical sets, contact loci with ideals and maximal divisorial sets  (see \cite{ELM} and \cite{Ishi08}). The following lemma gives us several additional classes of generically stable subsets of the arc space $X_\infty$. cf. \cite[Lemma 3.6]{Reg06}.
\begin{lemma}
Let $N$ be an irreducible  locally closed subset of $X_\infty$ with the generic point $z$. Then $N$ is generically stable if one of the following statements hold:
\begin{itemize}
\item[(i)] $N$ is semi-algebraic\footnote{see  \cite[(2.2)]{DL99}, \cite[3.4]{Reg06}}  and the corresponding morphism of $z$ is dominant;
\item[(ii)] there exists a resolution of singularities $h\colon Y\to X$ and a prime divisor $E$ of $Y$ such that $h_\infty$ maps the generic point of $\pi_Y^{-1}(E)$ to $z$.
\end{itemize}
\end{lemma}
\subsection{Reguera's Curve Selection Lemma}
\begin{theorem}\cite{Reg06}
Let $N$ and $ N'$ be irreducible locally closed subsets of  $X_\infty$ such that $\bar N\subsetneq N'$ and $N$ is generically stable. Let $z,z'$ be generic points of $N,N'$ respectively. Then there exists an arc 
$$\phi\colon \Spec K[[t]]\to N',$$
where $K$ is a finite algebraic extension of $k_z$, such that $\phi(0)=z$ and $\phi(\eta)\in N'\setminus N$.
\end{theorem}
\begin{corollary}
Assume $\mathrm{char}\ k=0$. Let $N$ be an irreducible subset of $X_\infty$ strictly contained in an irreducible component of $X^{\mathrm{Sing}}_\infty$ with the generic point $z$. Then, there exists an arc 
$$\phi\colon \Spec K[[t]]\to N',$$
where $K$ is a finite algebraic extension of $k_z$, such that $\phi(0)=z$ and $\phi(\eta)\in X^{\mathrm{Sing}}_\infty\setminus N$.
\end{corollary}

\begin{question}\cite[Page 127]{Reg06}\label{que}
Let $N$ and $ N'$ be irreducible locally closed subsets of  $X_\infty$ such that $\bar N\subsetneq N'$ and $N$ is generically stable. Let $z,z'$ be generic points of $N,N'$ respectively. Is it true that there exists an arc 
$$\phi\colon \Spec K[[t]]\to N',$$
where $K$ is a finite algebraic extension of $k_z$, such that $\phi(0)=z$ and $\phi(\eta)=z'$.
\end{question}
\subsection{Strong versions of the Curve Selection Lemma}\label{sec3}
In this section we prove several strong versions of Curve Selection Lemma. The first one answers affirmatively Reguera' question (Question \ref{que}).
\begin{theorem}\label{thm32}
Let $N$ and $ N'$ be irreducible locally closed subsets of  $X_\infty$ such that $\bar N\subsetneq N'$ and $N$ is generically stable. Let $z,z'$ be generic points of $N,N'$ respectively. Then there exists an arc 
$$\phi\colon \Spec K[[t]]\to N',$$
where $K$ is a finite algebraic extension of $k_z$, such that $\phi(0)=z$ and $\phi(\eta)=z'$.
\end{theorem}
\begin{proof}
Since $N$ is a generically stable of $X_\infty$, it follows from Lemma \ref{lm21} that the ring $\mathcal{O}_{N' ,z}$ is Noetherian. Applying the curve selection lemma for Noetherian $k_z$-schemes (Theorem \ref{thm12}), we obtain an arc defined by the following injective morphism of local $k_z$-algebras
$$ \mathcal{O}_{N' ,z} \to K[[t]],$$
$K$ is a finite algebraic extension of $k_z$. Hence the composition
$$ {\mathcal O}_{N'}\to  \mathcal{O}_{N' ,z} \to K[[t]]$$
defines an expected arc.
\end{proof}
In order to prove other strong versions of Curve Selection Lemma we need the following structure theorem, which generalizes Drinfeld-Grinberg-Kazhdan theorem \cite{GK00,Drin02,Drin20}. For its proof we need to use the proof of Drinfeld-Grinberg-Kazhdan theorem in \cite[Theorems 4.1-4.2]{BS17}.
\begin{lemma}\label{lemaa31}
Let $N$  be an  irreducible generically stable subset of  $X_\infty$. Let $\gamma\in N$ be a non-degenerate $k$-point  of $X_{\infty}$. Then there exists a local adic Noetherian $k$-algebra $A$  and an isomorphism 
$$ \widehat{\mathcal{O}}_{N,\gamma}\cong k[[\mathbb N]]\ \hat{\otimes}\ A,$$
where $ k[[\mathbb N]]$ stands for $ k[[x_1,x_2,\ldots, x_n,\ldots]]$.
\end{lemma}
\begin{proof}
As in the proofs of Drinfeld-Grinberg-Kazhdan theorem (see \cite{BS17,Drin02,Drin20}), we may assume that $X$ is a complete intersection, i.e. the subscheme of $\operatorname{Spec} k\left[x_1, \ldots, x_d, y_1, \ldots, y_l\right]$ defined by equations $f_1=\ldots=f_l=0$ such that the arc $\gamma_0(t)=\left(x^0(t), y^0(t)\right)$ is not contained in the subscheme of $X$ defined by det $\frac{\partial f}{\partial y}=0$. Here $\frac{\partial f}{\partial y}$ is the matrix of partial derivatives $\frac{\partial f_i}{\partial y_j}$. It follows from the proof of \cite[Theorem 4.1]{BS17}, \cite[Theorem 2.1.1]{Drin20} that there is an isomorphism 
$$\theta\colon \widehat{X_{\infty,\gamma}}\to \widehat{(\Bbb A^d_k)_{\infty,0}}\ {\times}\ \widehat{Y_y},$$
where $y$ is a $k$-point of some $k$-variety $Y$. Moreover, for each natural number $n$, there is a morphism 
$$\phi_n\colon \widehat{(\Bbb A^d_k)_{n,0}}\ {\times}\ \widehat{Y_y} \to \widehat{X_{n,\gamma_n}}$$ 
such that the following diagram commutes
\begin{displaymath}
\xymatrix{
 \widehat{N_{\gamma}}\ar@{->}[r]& \widehat{X_{\infty,\gamma}}\ar@{->}[d]_{\hat{\pi}_{n,X}}\ar@{->}[r]^{\theta}& \widehat{(\Bbb A^d_k)_{\infty,0}}\ {\times}\ \widehat{Y_y}\ar@{->}[d]^{p_{n}}\\
& \widehat{X_{n,\gamma_n}}&\widehat{(\Bbb A^d_k)_{n,0}}\ {\times}\ \widehat{Y_y}\ar@{->}[l]^{\phi_n}
}
\end{displaymath}
where $\gamma_n={\pi}_{n,X}(\gamma)$ and the vertical morphisms are induced by truncation maps. We take $n$  a positive integer such that $N\cap W$ is weakly stable at level $n$ for some open subset $W$ of $X_\infty$. Let $N_n$ be the closure of $\pi_n(N)$ in $X_n$. Since $N$ is generically stable, it follows that  $\widehat{\pi_n^{-1}(N_n)}_\gamma\cong  \widehat{N_{\gamma}}$. Then  the preimage of $\phi_n^{-1}\left(\widehat{N_{n,\gamma_n}}\right)$ is an affine formal subscheme of $\widehat{(\Bbb A^d_k)_{n,0}}\ {\times}\ \widehat{Y_y}$ and therefore 
$$\phi_n^{-1}\left(\widehat{N_{n,\gamma_n}}\right)=\Spf A,$$
for some local adic Noetherian $k$-algebra $A$. Since $p_n$ is a trivial fibration with fiber $\Spf(k[[\Bbb N]])$, it yields that
 $$ \Spf \widehat{\mathcal{O}}_{N,\gamma}\cong \hat{\pi}_{n,X}^{-1}\left( \phi_n(\Spf A)\right)= \theta^{-1}\left( p^{-1}_n(\Spf A)\right)\cong \Spf k[[\mathbb N]]\ \times \Spf A$$
and hence
$$\widehat{\mathcal{O}}_{N,\gamma}\cong k[[\mathbb N]]\ \hat{\otimes} A.$$
\end{proof}
\begin{remark}
By a more concrete argument we may indeed choose the $k$-algebra $A$ in the statement of Theorem \ref{main} such that $\Spf A$ is the completion of a $k$-variety at a $k$-point. Nevertheless, we do not need such a strong result in this paper. 
\end{remark}
\begin{theorem}\label{main}
Let $N$ be an irreducible  generically stable subset of  $X_\infty$ with the generic point $z$. Let $\gamma\in N$ be a non-degenerate $k$-point. Then there exist an extension $K$ of $k$ and an arc 
$$\phi\colon \Spec K[[t]]\to N,$$
such that $\phi(0)=\gamma$ and $\phi(\eta)=z$.
\end{theorem}
\begin{proof}
Since $\gamma$ is a non-degenerate $k$-arc, it follows from Lemma \ref{lemaa31} that there is an isomorphism of $k$-algebras
$$ \widehat{\mathcal{O}}_{N,\gamma}\cong k[[\mathbb N]]\ \hat{\otimes}\ A,$$
where $ k[[\mathbb N]]$ stands for $ k[[x_1,x_2,\ldots, x_n,\ldots]]$ and $A$ is a local adic Noetherian $k$-algebra. Applying Theorem \ref{thm12}, we get an arc defined by the following injective morphism of local $k$-algebras
$$A\to K_1[[t]].$$
We denote by $K_2$ the quotient field of the integral domain $k[[\mathbb N]]$ and by $K$ the completed tensor product $K_1  \hat{\otimes}\  K_2 $. Let $\phi$ be the arc defined by the following composition of injective morphisms
$$ {\mathcal O}_{N}\to   \widehat{\mathcal{O}}_{N,\gamma} \cong  k[[\mathbb N]]\ \hat{\otimes}\  A  \cong  k[[\mathbb N]]\ \hat{\otimes}\ K_1[[t]] \to K[[t]].$$
Then $\phi(0)=\gamma$ and $\phi(\eta)=z$.
\end{proof}
The following example shows that the assumption that $N$ is generically stable in Theorem \ref{thm32} and the assumption that $\gamma$ is non-degenerate in Theorem \ref{main} are necessary. 
\begin{example}[Lejeune-Jalabert and Reguera]
Let $X$ be the Whitney umbrella $x_3^2=x_1 x_2^2$ in $\mathbb{A}_{\mathbb{C}}^3$. Then $\mathrm{Sing }X$ is defined by $x_2=x_3=0$. Let $\gamma$ be the point in $X_{\infty}$ determined by any arc $x_1(t), x_2(t), x_3(t)$ such that
$$
\operatorname{ord}_t x_1(t)=1 \quad \text { and } \quad x_2(t)=x_3(t)=0 .
$$
Let $N$ be the closure of the point $\gamma$ and let $N^{\prime}$ be the set $\pi_X^{-1}(\mathrm{Sing }X) \backslash(\operatorname{Sing} X)_{\infty}$, the set of arcs centered in some point of $\operatorname{Sing} X$. Then $N \subset N^{\prime}$ (\cite[Lemma 2.12]{IK03}) but there does not exist an arc $\phi\colon \Spec K[[s]] \rightarrow N^{\prime}$ which maps the closed point to $\gamma$ and the generic point to the generic point of $N^{\prime}$. 

In fact, assume that such an arc exists, i.e. there is a wedge whose coordinates 
$$x_1(t,s), x_2(t,s), x_3(t,s) \in K[[t, s]]$$ satisfy $x_3^2=x_1 x_2^2$; 
$$x_1(t,0)=x_1(t); x_2(t,0)=x_2(t)=0\text{ and }x_3(t,0)=x_3(t)=0.$$ Then $\operatorname{ord}_{(t,s)} x_1(t,s)=1$ and thus 
$$2\operatorname{ord}_{(t,s)} x_3(t,s)=1+2 \operatorname{ord}_{(t,s)} x_2(t,s).$$
Hence $x_2(t,s)$ and $x_3(t,s)$ must be equal to zero, i.e. the image of the generic poit of $\Spec K[[s]]$ is in $(\operatorname{Sing} X)_{\infty}$, a contradiction.
\end{example}
Notice that the output of the previous result is a parametrization defined over the field $K$ which is of infinite transcendence degree over the base field. In many applications it is necessary to obtain a Curve Selection Lemma whose outcome curve is defined over the base field.
\begin{corollary}\label{lastcoro}
Let $N$ be an irreducible  generically stable subset of  $X_\infty$ with the generic point $z$. Let $P$ is another irreducible closed subset of $X_\infty$ not containing $N$. Let $\gamma\in N$ be a non-degenerate $k$-point. Then there exisst an arc 
$$\phi\colon \Spec k[[t]]\to N,$$
which maps the closed point to $\gamma$ and the generic point outside $P$.
\end{corollary}
\begin{proof}
It is proved in the same way as in the proof of Theorem \ref{main} by using a cutting method (cf. the proof of Corollary \ref{coro22}).
\end{proof}

\end{document}